\documentclass[12pt,xcolor=svgnames]{amsart}

\usepackage{amsmath, amsthm, amssymb, amsfonts, verbatim}

\usepackage{xspace,mathtools,bm}

\usepackage[svgnames]{xcolor}

\usepackage[pagebackref=true, colorlinks=true, citecolor=blue]{hyperref}

\allowdisplaybreaks

\DeclareRobustCommand{\rchi}{{\mathpalette\irchi\relax}}
\newcommand{\irchi}[2]{\raisebox{\depth}{$#1\chi$}}

\newtheorem{theorem}{Theorem}[section]

\newtheorem{prop}[theorem]{Proposition}

\newtheorem{lemma}[theorem]{Lemma}

\theoremstyle{definition}
\newtheorem{definition}[theorem]{Definition}

\newtheorem*{theorem*}{Theorem}

\theoremstyle{remark}
\newtheorem{remark}[theorem]{Remark}

\newcommand{\Obsolete}[1]{
    }

\usepackage[english]{babel}

\begin{document}

\raggedbottom

\numberwithin{equation}{section}

\title
    [Euler equations in a critical Sobolev space]
    {Local Existence for the 2D Euler Equations in a Critical Sobolev Space}

\author{Elaine Cozzi}
\address{Oregon State University, Department of Mathematics}
\curraddr{}
\email{elaine.cozzi@oregonstate.edu}

\author{Nicholas Harrison}
\address{Oregon State University, Department of Mathematics}
\curraddr{}
\email{harrnich@oregonstate.edu}
\begin{abstract}
In the seminal work \cite{bourgain}, Bourgain and Li establish strong ill-posedness of the 2D incompressible Euler equations with vorticity in the critical Sobolev space $W^{s,p}(\mathbb{R}^2)$ for $sp=2$ and $p\in(1,\infty)$. In this note, we establish short-time existence of solutions with vorticity in the critical space $W^{2,1}(\mathbb{R}^2)$. Under the additional assumption that the initial vorticity is Dini continuous, we prove that the $W^{2,1}$-regularity of vorticity persists for all time. 
\end{abstract}

\maketitle

\section{Introduction}
In this paper, we investigate the persistence of regularity of solutions to the two-dimensional incompressible Euler equations, given by 
\begin{align*} \label{E}\tag{E}
    \begin{cases}
        \partial_t u +(u\cdot\nabla) u=-\nabla p &\text{in }[0,T]\times \mathbb{R}^2,
        \\ \nabla\cdot u=0 &\text{in }[0,T]\times \mathbb{R}^2, 
        \\ u\big|_{t=0}=u_0 &\text{in }\mathbb{R}^2,
    \end{cases}
\end{align*}
where $u:[0,T]\times\mathbb{R}^2\to\mathbb{R}^2$ represents the velocity field of an ideal fluid and $p:[0,T]\times\mathbb{R}^2\to\mathbb{R}$ denotes the scalar pressure. We define the \textit{vorticity} as the curl of the velocity, i.e. $\omega=\text{curl}(u)=\partial_{1}u_2-\partial_{2}u_1$.  Applying the curl operator to (\ref{E}), we obtain the vorticity equation,
\begin{equation*} \label{V}\tag{V}
\begin{cases}
    \partial_t\omega+u\cdot\nabla \omega=0 &\text{in }[0,T]\times \mathbb{R}^2,
    \\ u=\nabla^\perp(-\Delta)^{-1}\omega=K*\omega &\text{in }[0,T]\times \mathbb{R}^2,
    \\ \omega\big|_{t=0}=\omega_0 &\text{in }\mathbb{R}^2,
\end{cases}    
\end{equation*} 
where the relation $u=K\ast\omega$ is referred to as the Biot-Savart Law, with $$K(x)=\frac{1}{2\pi}\frac{x^\perp}{|x|^2}$$ the Biot-Savart kernel.  We refer to convolution with $K$ as the Biot-Savart operator. 

Global existence of weak solutions to (\ref{V}) with vorticity in $L^1\cap L^{\infty}(\mathbb{R}^2)$ was established by Yudovich in \cite{Yudovich}. In what follows, we are interested in the persistence of regularity of Yudovich weak solutions under the assumption that initial vorticity belongs to the critical Sobolev space $W^{2,1}(\mathbb{R}^2)$. 

The notion of criticality can be understood in view of the following estimate for (\ref{V}), which one can establish using standard techniques:  
\begin{equation}\label{criticalinequality}
\frac{d}{dt}\|\omega(t)\|_{W^{s,p}}\leq C\|\nabla u(t)\|_\infty \|\omega(t)\|_{W^{s,p}}
\end{equation}
for some dimensional constant $C>0$ independent of $\omega_0\in W^{s,p}(\mathbb{R}^2)$. In order to close this inequality and obtain short-time existence of solutions, a bound of the form $\|\nabla u(t)\|_\infty\leq C \|\omega(t)\|_{W^{s,p}}$ is required. When $sp>2$, this bound follows from Sobolev embeddings and boundedness properties of the Biot-Savart operator. In fact, in this setting, one can establish a stronger Beale-Kato-Majda \cite{bkm} type of estimate, in which the velocity gradient grows logarithmically with respect to the Sobolev norm of the vorticity, yielding global existence of solutions. When $sp\leq 2$, however, the situation is considerably more complicated. For example, when $sp=2$ and $p\in(1,\infty)$, the estimate $\|\nabla u(t)\|_\infty\leq C \|\omega(t)\|_{W^{s,p}}$ just barely fails to hold. Consequently, we call the sets of parameters $(s,p)$ supercritical, critical, and subcritical when they satisfy $sp> 2$, $sp=2$, and $sp< 2$, respectively. 
\Obsolete{One can also understand criticality in terms of scaling, in the sense that the $L^p$ norm of the highest order derivatives possess the same scaling as that of (\ref{V}). Specifically, note that if $\omega$ solves (\ref{V}) and we define the transformation $$\omega_\lambda(t,x):=\omega( t,\lambda x),$$ then $\omega_{\lambda}$ also solves (\ref{V}). It is easy to check, for example, that $\| \nabla\omega_\lambda \|_{L^2} = \| \nabla\omega \|_{L^2}$ and that $\| \nabla^2\omega_\lambda \|_{L^1} = \| \nabla^2\omega \|_{L^1}$.}

In the supercritical regime $sp>2$ with $1<p<\infty$, Kato and Ponce \cite{katoponce} obtain global-in-time well-posedness of the vorticity equation (\ref{V}) in $W^{s,p}(\mathbb{R}^2)$. The boundedness of the Biot-Savart operator on $L^p(\mathbb{R}^2)$ for $p\in (1,\infty)$ implies an equivalence of the Kato-Ponce result to global well-posedness of (\ref{E})
in $W^{s+1,p}(\mathbb{R}^2)^2$ for the same range of $s,p$. The critical regime $sp=2$ was mostly settled in the seminal work of Bourgain and Li \cite{bourgain}, where the authors establish strong ill-posedness of (\ref{V}) when $1<p<\infty$. In \cite{bourgain}, a lack of an a priori bound on the Lipschitz norm of the velocity is key in constructing a perturbation that grows arbitrarily fast in the $W^{s,p}$-norm.  For the subcritical case $sp<2$, Bahouri and Chemin \cite{BC} establish a lower bound on the decay rate with respect to time of smoothness to the solution to (\ref{V}) with initial data in $W^{s,p}(\mathbb{R}^2)$.  The authors also provide an example demonstrating that their decay rate is sharp.

To our knowledge, the critical case $(s,p)=(2,1)$ remains unresolved. As one might expect, difficulties arise when $p=1$ due to non-reflexivity of the $L^1$-based Sobolev spaces and to the unboundedness of singular integral operators on $L^1$. On the other hand, the Sobolev embedding theorem gives $$W^{s,p}(\mathbb{R}^2)\hookrightarrow W^{s', p'}(\mathbb{R}^2)$$   
for $sp = s'p' = 2$ and $1\leq s' \leq s \leq 2$, implying that $W^{2,1}(\mathbb{R}^2)$ is the most regular of the critical Sobolev spaces. Indeed, $W^{2,1}(\mathbb{R}^2)$ functions are uniformly continuous (see Remark \ref{dini}), while $W^{s,p}(\mathbb{R}^2)$ functions for $s<2$ and $sp=2$ may have singularities.

Some well-posedness results have been obtained for the 2D Euler equations in $L^1$-based spaces. Chae \cite{chae} showed global well-posedness of (\ref{E}) in the related $L^1$-based scaling critical Triebel-Lizorkin spaces $F^3_{1,q}$ for $q\in[1,\infty]$. In particular, the case $q=2$ corresponds to a higher regularity local Hardy space which embeds into $W^{3,1}(\mathbb{R}^2)$; however, $F^3_{1,q}$ does not coincide with $W^{3,1}$ for any value of $q$.  We also refer the reader to \cite{crippa2017} and \cite{bohun2016} for existence results in a lower regularity, $L^1$-based setting.

Before stating our main result, we mention that following the work of Bourgain-Li \cite{bourgain}, several additional ill-posedness results for (\ref{E}) were obtained. See, for example, \cite{EJ} for a construction of solutions whose vorticity exhibits instantaneous blow-up in $H^1(\mathbb{R}^2)$. Also see  \cite{BL2, EM} for strong ill-posedness when initial velocity belongs to $C^k(\mathbb{R}^2)$, $k\geq 1$ an integer.    

Our main result addresses existence of a weak solution (see Definition \ref{Weakvorticity}) in the critical case $(s,p)=(2,1)$. 
 \begin{theorem}\label{Shorttimesec1} Given $\omega_0\in W^{2,1}(\mathbb{R}^2)$, there exists $T>0$ and a weak solution $\omega$ to (\ref{V}) on $[0,T]$ satisfying $$\omega\in L^\infty([0,T];W^{2,1}(\mathbb{R}^2))\cap \text{Lip}([0,T];W^{1,1}(\mathbb{R}^2)).$$
\end{theorem}
\noindent We prove Theorem \ref{Shorttimesec1} in Section \ref{ProofOfMainTheorem}.

Due to lack of boundedness of Calderon-Zygmund operators on $L^1(\mathbb{R}^2)$, we cannot immediately conclude from Theorem \ref{Shorttimesec1} that (\ref{E}) is locally well-posed in $W^{3,1}(\mathbb{R}^2)$. However, one can utilize the equivalence $W^{s,p}=B^{s}_{p,p}$ for $p>1$ or $p=1$ and $s\notin \mathbb{Z}$ (see \cite{triebel}), where $B^{s}_{p,q}$ denotes the Besov space, to obtain an arbitrarily small loss of regularity of the velocity with initial data in $W^{3,1}(\mathbb{R}^2)^2$. We do not include the details here; instead, we refer the reader to the recent work \cite{pakhwang}, which demonstrates this arbitrarily small loss of regularity by establishing short time existence (global existence when $d=2$) of velocity solutions in the Triebel-Lizorkin spaces $F^{d+1}_{1,\infty}(\mathbb{R}^d)$. These spaces lie between the spaces $W^{d+1,1}(\mathbb{R}^d)$ and $W^{d+1-\varepsilon,1}(\mathbb{R}^d)$ for any $\varepsilon>0$.  
\Obsolete{We have the following result, which we also prove in Section \ref{section3}. 
\begin{theorem}\label{velocitycorollary}
    Let $\epsilon>0$. Suppose $u_0\in W^{3,1}(\mathbb{R}^2)^2$. There exists $T>0$ and $u\in L^\infty([0,T];W^{3-\varepsilon}(\mathbb{R}^2)^2)$ such that $u$ solves ($E$) on $[0,T]\times \mathbb{R}^2.$
\end{theorem}
 \noindent Theorem \ref{velocitycorollary} also follows from }
 
 In Section \ref{section4}, we show that under an additional Dini continuity assumption on the initial vorticity, the vorticity solution to (\ref{V}) remains in $W^{2,1}(\mathbb{R}^2)$ for all time (see Theorem \ref{Globaltheorem}). The additional continuity assumption allows us to control the velocity gradient (see Proposition \ref{velgraddiniprop}), which yields global existence. Similar calculations are given in \cite{koch} on a bounded domain. We follow \cite{koch}, but modify the arguments to fit our setting. 
 
 \Obsolete{We have the following theorem.  
 \begin{theorem}\label{Globaltheorem} If $\omega_0\in C_D\cap W^{2,1}(\mathbb{R}^2)$, then the solution $\omega$ obtained in Theorem \ref{Shorttimesec1} belongs to $L^\infty_{loc}(\mathbb{R}_+;C_D\cap W^{2,1} (\mathbb{R}^2))$. Furthermore, $\|\omega(t)\|_{W^{2,1}\cap C_D}:=\|\omega(t)\|_{W^{2,1}}+|\omega(t)|_{C_D}$ grows at most doubly exponentially with $t$.
\end{theorem}
  \noindent We remark that the Dini continuity assumption in Theorem \ref{Globaltheorem} is stronger than the uniform continuity already satisfied by $W^{2,1}(\mathbb{R}^2)$ functions (see Remark \ref{dini} below), but is weaker than an $\alpha-$H\"older condition with $\alpha>0$, for which global existence is established in \cite{wolibner1933} (see \cite{chemin} for another approach). }
  
  While our results shed light on the short-time behavior of solutions to (\ref{V}) with initial vorticity in $W^{2,1}(\mathbb{R}^2)$, the question of global existence in this space remains unsolved. In particular, without a sufficiently nice modulus of continuity on the initial vorticity, our techniques do not yield control of the $L^{\infty}$-norm of the velocity gradient by norms of the initial data, as in Proposition \ref{velgraddiniprop}, or via a logarithmic estimate of Beale-Kato-Majda type \cite{bkm}. Our techniques also fail to yield existence, even for short time, to the 3D Euler equations or the 2D quasi-geostrophic equations with vorticity in the scaling critical space $W^{3,1}$. We plan to address well-posedness for these equations in a future work.
  
  The paper is organized as follows. In Section \ref{section2}, we state some definitions, theorems, and lemmas which will be useful in the subsequent sections. In Section \ref{section3}, we prove the key a priori estimate on solutions to (\ref{V}) with initial vorticity in $W^{2,1}(\mathbb{R}^2)$. In Section \ref{ProofOfMainTheorem}, we prove Theorem \ref{Shorttimesec1}. Finally, in Section \ref{section4}, we establich global existence under an additional Dini continuity assumption on the initial vorticity. 
%
%
%
\section{Notation and Preliminaries}\label{section2}

\subsection{Notation} 
In what follows, we let $B_r(x)$ denote the ball of radius $r>0$ about the point $x\in\mathbb{R}^d$.  We let $|A|$ denote the Lebesgue measure of a set $A\subset\mathbb{R}^d$. Finally, we let $C$ denote a generic positive constant which may change from line to line, and its dependence on any important quantities will be specified.
\begin{definition}
Let $k$ be a non-negative integer. The {\em Sobolev space} $W^{k,p}(\mathbb{R}^d)$ is the space of all distributions $f$ in $L^p(\mathbb{R}^d)$ whose derivatives of all orders up to and including $k$ are also in $L^p(\mathbb{R}^d)$.  We equip $W^{k,p}(\mathbb{R}^d)$ with the norm $$\|f\|_{W^{k,p}}=\|f\|_{L^p}+\sum_{|\sigma|=1}^k\|D^\sigma f\|_{L^p}.$$
\end{definition} 
\subsection{ODE lemmas}
 We repeatedly use the following two lemmas. The first lemma is the classical Gr\"onwall inequality, while the second is the more general Osgood's lemma (see \cite{chemin}, chapter 5 for a proof).
\begin{lemma}\label{Gronwall}
     Let $T>0$ and let $L,\alpha,\beta$ be nonnegative, continuous functions on $[0,T]$ with $\alpha$ nondecreasing. If for all $t\in[0,T]$, $$L(t)\leq \alpha(t)+\int_0^t \beta(s)L(s) \, ds,$$ then for all $t\in[0,T]$, $$L(t)\leq \alpha(t)\exp{\left(\int_0^t\beta(s)\,ds\right)}.$$ 
\end{lemma}
\noindent We now state Osgood's Lemma.
\begin{lemma}\label{Osgood}
     Let $T>0 $, $\rho$ a positive borelian function, $\gamma$ a locally integrable positive function, and $\mu$ a continuous increasing function. Assume that for some number $\beta>0$ and for all $t\in[0,T]$, these functions satisfy $$\rho(t)\leq\beta +\int_0^t\gamma(s)\mu(\rho(s))\,ds.$$
     Then for all $t\in[0,T]$ $$-\phi(\rho(t))+\phi(\beta)\leq \int_0^t \gamma(s)\,ds,$$ where $\phi(x)=\int_x^1\frac{1}{\mu(r)}\,dr$. 
\end{lemma}
\subsection{Compactness properties for $L^1(\mathbb{R}^2)$.} To prove Theorem 
\ref{Shorttimesec1}, we will construct a sequence of smooth solutions to (\ref{V}) and show that our sequence is uniformly bounded in $W^{2,1}(\mathbb{R}^2)$, which will allow us to pass to a weak limit and obtain a candidate for a weak solution. Some care is needed, however, when passing to the limit, due to non-reflexivity of $L^1(\mathbb{R}^2)$. We therefore need a few results on the strong and weak compactness properties of $L^1(\mathbb{R}^2)$. First, we state a few definitions.  
 \begin{definition}\label{props}
    Let $p\in[1,\infty)$ and let $\mathcal{F}\subset L^p(\mathbb{R}^d)$. We call $\mathcal{F}$ \textit{equitight} in $L^p(\mathbb{R}^d)$ if \begin{equation}
        \tag{T} \lim_{R\to\infty}\sup_{f\in\mathcal{F}}\int_{|x|\geq R}|f(x)|^p\, dx=0.
    \end{equation} 
    $\mathcal{F}$ is called \textit{equicontinuous} in $L^p(\mathbb{R}^d)$ if \begin{equation}
        \tag{C}\lim_{h\to0}\sup_{f\in\mathcal{F}}\int_{\mathbb{R}^d}|f(x)-f(x+h)|^p \, dx=0.
    \end{equation}
    $\mathcal{F}$ is called \textit{equiintegrable} in $L^p(\mathbb{R}^d)$ if \begin{equation}
        \tag{I} \lim_{\delta\to0}\sup_{|A|<\delta}\sup_{f\in\mathcal{F}}\int_{A}|f(x)|^p \, dx=0.
    \end{equation}
\end{definition}
\noindent The properties in Definition \ref{props} allow us to neatly characterize weakly and strongly precompact sets in $L^1(\mathbb{R}^d)$. 
The following result is the \textit{Dunford-Pettis} Theorem \cite{royden}.

\begin{theorem}\label{DP}
    A collection $\mathcal{F}\subset L^1(\mathbb{R}^d)$ is precompact in the weak topology if and only if $\mathcal{F}$ is equiintegrable and equitight in $L^1(\mathbb{R}^d)$
\end{theorem}

For strong compactness in $L^1(\mathbb{R}^d)$, we have the \textit{Fr\'echet-Kolmogorov-Riesz} theorem, see \cite{Brezis}.

\begin{theorem}\label{RFK}
    A collection $\mathcal{F}\subset L^1(\mathbb{R}^d)$ is precompact in the strong topology if and only if $\mathcal{F}$ is equicontinuous and equitight in $L^1(\mathbb{R}^d)$. 
\end{theorem}

\subsection{Lorentz spaces and velocity bounds} Lorentz spaces will provide us with the optimal embeddings to ensure a bound on the velocity gradient and close the inequality $(\ref{criticalinequality})$. To define these spaces, given $f:\mathbb{R}^d\to\mathbb{R}$, we first define the {\em distribution function} $$d_f(\alpha)=\big|\{x\in\mathbb{R}^d:|f(x)|> \alpha\}|$$ and the \textit{decreasing rearrangement} of $f$, $f^*:[0,\infty)\to[0,\infty]$, $$f^*(t)=\inf\{s>0:d_f(s)\leq t\}.$$
\begin{definition}
For $p\in[1,\infty)$ and $q\in[1,\infty]$, define the {\em Lorentz space} $L^{(p,q)}(\mathbb{R}^d)$ to be the space of measurable functions (identified when differing only on a set of Lebesgue measure zero) satisfying $\|f\|_{L^{(p,q)}}<\infty$, where 
$$\|f\|_{L^{(p,q)}}=\begin{cases}
    \bigg(\displaystyle\int_0^\infty \left(t^{\frac{1}{p}}f^*(t)\right)^q\frac{dt}{t}\bigg)^{\frac{1}{q}}\quad& \text{if }q<\infty,
    \\ \left(\sup_{t>0} t^pf^*(t)\right)^{\frac{1}{p}} \quad &\text{if }q=\infty.
\end{cases}$$
Under this norm, $L^{(p,q)}(\mathbb{R}^d)$ is complete. The Lorentz spaces satisfy the embedding 
\begin{equation*}
L^{(p,q)}\hookrightarrow L^{(p,r)}
\end{equation*}
for $q<r$. Moreover, $L^{(p,p)}=L^p$ for all $p\in[1,\infty)$, so that
\begin{equation}\label{LorentzLebesgue}
L^{(p,q)}\hookrightarrow L^{p}
\end{equation}
for all $q<p$.
\end{definition}
We remark that Lorentz spaces satisfy sharper embeddings than those given by the classical Sobolev embedding theorem. To see this, note that for all $f\in W^{1,1}(\mathbb{R}^d)$,
\begin{equation}\label{sobolevlorentz}
\|f\|_{L^{\left(\frac{d}{d-1},1\right)}}\leq C\|\nabla f\|_{L^1},
\end{equation}
where $C>0$ is a dimensional constant independent of $f$ (see \cite{tartar} for a proof of (\ref{sobolevlorentz})). This inequality implies the continuous embedding $$W^{1,1}(\mathbb{R}^2)\hookrightarrow L^{(2,1)}(\mathbb{R}^2),$$ which, in view of (\ref{LorentzLebesgue}), is sharper than the classical Sobolev embedding $W^{1,1}(\mathbb{R}^2)\hookrightarrow L^{2}(\mathbb{R}^2)$. 

Lorentz spaces also satisfy a duality relation that allows us to extend the standard H\"older inequality. We have, for $p\in(1,\infty)$ and $q\in[1,\infty)$, the relation $$\Big(L^{(p,q)}(\mathbb{R}^d)\Big)^*=L^{(p^*,q^*)}(\mathbb{R}^d), \quad \frac{1}{p}+\frac{1}{p^*}=\frac{1}{q}+\frac{1}{q^*}=1,$$ where $q^*=\infty$ if $q=1$. In particular, we obtain the H\"older-Lorentz inequality 
\begin{equation}\label{holderlorentz}
    \int_{\mathbb{R}^d}f(x)g(x)\, dx\leq \|f\|_{L^{(2,1)}}\|g\|_{L^{(2,\infty)}},
\end{equation}
for $f\in L^{(2,1)}(\mathbb{R}^d)$ and $g\in L^{(2,\infty)}(\mathbb{R}^d).$

We are now in a position to prove the key bound on the $L^{\infty}$-norm of the velocity gradient when vorticity is in $W^{2,1}(\mathbb{R}^2)$.  This bound will allow us to close the inequality (\ref{criticalinequality}) when proving the a priori estimate in Section \ref{section3}.
\begin{lemma}\label{velgradbound}
Assume that $u$ is a smooth, divergence free vector field on $\mathbb{R}^2$ with vorticity $\omega$ in $W^{2,1}(\mathbb{R}^2)$. Then 
\begin{equation*}
\| \nabla u \|_{L^{\infty}} \leq C\| \omega \|_{W^{2,1}}.
\end{equation*}
\end{lemma}
\begin{proof}
From the Biot-Savart law, we have the identity $$\nabla u=K*\nabla\omega +c\omega I$$ for some constant $c>0$ and for $I$ the $2\times 2$ identity matrix (see \cite{majda}). Because the Lorentz norms are rearrangement invariant, we have 
\begin{equation}\label{nablabound}
\begin{split}
    &|\nabla u(x)| \leq \int_{\mathbb{R}^2} \frac{1}{2\pi |x-y|}\big|\nabla\omega(y)\big|\, dy+c|\omega(x)| \\
    &\qquad \leq \left\|\frac{1}{2\pi |\cdot|}\right\|_{L^{(2,\infty)}}\|\nabla \omega\|_{L^{(2,1)}}+c\|\omega\|_\infty
    \\ &\qquad \leq C\|\nabla \omega\|_{W^{1,1}}+c\|\omega\|_\infty
    \leq C\big(\|\omega\|_{W^{2,1}}+\|\omega\|_\infty\big),
    \end{split}
\end{equation} where we applied (\ref{holderlorentz}) to obtain the second inequality and (\ref{sobolevlorentz}) to get the third inequality. Now observe that for smooth and compactly supported $f\in W^{2,1}(\mathbb{R}^2)$, one has $$f(x_0,y_0)=\int_{-\infty}^{y_0}\int_{-\infty}^{x_0}\partial_{1}\partial_2f(x,y)\, dx\, dy,$$
and hence $\|f\|_\infty\leq \|\partial_{1}\partial_2f\|_{L^1}$. The general case follows by approximation using the density of $C^\infty_0(\mathbb{R}^2)$ in $W^{2,1}(\mathbb{R}^2)$. We conclude that
$$\|\omega\|_\infty\leq \|\omega\|_{W^{2,1}}.$$
Substituting this estimate into (\ref{nablabound}) gives 
\begin{equation*}
    \|\nabla u\|_\infty\leq C\|\omega\|_{W^{2,1}}.
\end{equation*}
\end{proof}
\begin{remark}\label{Bourgain}
Lemma \ref{velgradbound} also follows from techniques used to prove Theorem 3 in \cite{BB}.
\end{remark}
\begin{remark}\label{dini}
Note that the estimate $\|f\|_\infty\leq \|\partial_{1}\partial_2f\|_{L^1}$ implies that $W^{2,1}(\mathbb{R}^2)$ functions can be uniformly approximated by $C^\infty_0(\mathbb{R}^2)$ functions, showing that they are, in fact, uniformly continuous. 
\end{remark}

\subsection{Weak Solution to (\ref{V})}
In what follows, we use the following definition of a weak solution to (\ref{V}). Global existence of such solutions is due to Yudovich \cite{Yudovich}.
\begin{definition}\label{Weakvorticity}
    Let $T>0$. A pair $(\omega,u)$ is called a {\em weak solution} to the vorticity equation (\ref{V}) on $[0,T]$ with initial data $\omega_0$ if $u$ and $\omega$ satisfy:
    \begin{enumerate}
        \item $\omega\in L^\infty([0,T];L^1\cap L^\infty(\mathbb{R}^2))$,
        \medskip
        \item $u=K*\omega$ on $[0,T]\times\mathbb{R}^2$, where $K$ is the Biot-Savart kernel, and 
        \medskip
        \item for all $\varphi\in C^1_0([0,T]\times\mathbb{R}^2)$ the following identity holds:
        $$\int_{\mathbb{R}^2}\varphi(T)\omega(T)-\varphi(0)\omega_0\,dx=\int_0^T\int_{\mathbb{R}^2}\big(\partial_t\varphi+u\cdot\nabla\varphi\big)\omega \,dx\,dt.$$
    \end{enumerate}
\end{definition} 
\section{A Priori Estimate}\label{section3}
To prove Theorem \ref{Shorttimesec1}, we must first establish a priori control of the $W^{2,1}$-norm of smooth solutions to (\ref{V}) for short time. We have the following proposition.
 \begin{prop}\label{Eulerapriori}
 Assume $\omega$ is a smooth solution to (\ref{V}) with initial data $\omega_0\in W^{2,1}(\mathbb{R}^2)$. Then there exists a constant $C>0$ independent of $\omega_0$ and a time $T>0$ such that for all $t\in [0,T]$, $$\|\omega(t)\|_{W^{2,1}}\leq\frac{\|\omega_0\|_{W^{2,1}}}{1-Ct\|\omega_0\|_{W^{2,1}}}.$$
\end{prop} 

\noindent \begin{proof}
 Define the particle trajectory map $X:[0,\infty)\times\mathbb{R}^2
\to\mathbb{R}^2$ by
\begin{equation}\label{flowode}
    \frac{dX}{dt}(t,\alpha)=u(t,X(t,\alpha)),\quad X(0,\alpha)=\alpha,
\end{equation}
 where $u=K*\omega$ is the velocity field corresponding to $\omega$. Since $u$ is divergence-free and smooth, $X(t,\cdot)$ is a measure-preserving diffeomorphism of $\mathbb{R}^2$ to itself. Using the chain rule, we may write (\ref{V}) as $$\partial_t\big(\omega(t,X(t,\alpha))\big)=0.$$ Integrating in time yields \begin{equation}\label{characteristics}
\omega(t,X(t,\alpha))=\omega_0(\alpha).
\end{equation} 
Since $X(t,\cdot)$ is measure-preserving, it follows that
\begin{equation}\label{Lpcons}
    \| \omega (t)\|_{L^p}=\|\omega_0\|_{L^p}
\end{equation} 
for all $p\in[1,\infty]$.\\
\\
Differentiating (\ref{V}) in space and using (\ref{flowode}) yields for $i=1,2$,
\begin{align*}
    \partial_t\big(\partial_{i}\omega(t,X(t,\alpha))\big) &=\partial_t\partial_{i}\omega(t,X(t,\alpha))+u(t,X(t,\alpha))\cdot\nabla\partial_{i}\omega(t,X(t,\alpha))
    \\&=-\partial_{i}u(t,X(t,\alpha))\cdot\nabla\omega(t,X(t,\alpha)).
\end{align*} 
We integrate in time, use that $X$ is measure preserving, and apply H\"older's inequality. This gives \begin{equation}\label{onederiv}
\begin{split}
&\|\partial_{i}\omega(t)\|_{L^1}\leq\|\partial_{i}\omega_0\|_{L^1}  +\int_0^t\|\partial_{i}u(s)\cdot\nabla\omega(s)\|_{L^1}\, ds\\
  &\qquad \leq\|\partial_{i}\omega_0\|_{L^1}  +\int_0^t\|\nabla u(s)\|_\infty\|\nabla\omega(s)\|_{L^1} \, ds.
  \end{split}
\end{equation} 
Taking two derivatives in space with respect to the directions $i$ and $j$, where $i,j\in\{1,2\}$, we similarly obtain \begin{equation}\label{twoderivprep} 
\begin{split}
&\|\partial_{ij}^2\omega(t)\|_{L^{1}}\leq\|\partial_{ij}^2\omega_0\|_{L^1}+\int_0^t \|\partial_{ij}^2u(s)\cdot\nabla\omega(s)\|_{L^1} \, ds\\
&+\int_0^t(\|\partial_{j}u(s)\cdot\nabla\partial_{i}\omega(s)\|_{L^1}
+\|\partial_{i}u(s)\cdot\nabla\partial_{j}\omega(s)\|_{L^1}) \,ds.
    \end{split}
\end{equation} 
We estimate the terms under the time integral in (\ref{twoderivprep}).  First, note that
\begin{equation}\label{pretermone}
\begin{split}
&\|\partial_{ij}^2u(s)\cdot\nabla\omega(s)\|_{L^1} \leq \|\partial_{ij}^2u(s)\|_{L^2}\|\nabla\omega(s)\|_{L^2} 
    \\ &\qquad \leq C\|\partial_{j}\omega(s)\|_{L^2}\|\nabla\omega(s)\|_{L^2} \leq C\|\nabla\omega(s)\|_{L^2}^2,    
\end{split}
\end{equation}
where we used H\"older's inequality and the $L^2$-boundedness of the operator $\nabla K*\cdot$. Integration by parts yields \begin{equation}\label{termone}
\begin{split}
&\|\partial_{ij}^2u(s)\cdot\nabla\omega(s)\|_{L^1}\leq C\int_{\mathbb{R}^2}|\nabla\omega(s,x)|^2 \, dx
    = -C\int_{\mathbb{R}^2}\omega(s,x)\Delta\omega(s,x) \, dx
    \\  &\leq C\|\omega(s)\|_{W^{2,1}}\|\omega(s)\|_\infty
    \leq C\|\omega(s)\|_{W^{2,1}}\|\nabla u(s)\|_\infty.
\end{split}
\end{equation}
The remaining terms under the time integral in (\ref{twoderivprep}) can be estimated using H\"older's inequality.  We write  
\begin{equation}\label{termtwo}
\begin{split}
&\|\partial_{i}u(s)\cdot\nabla\partial_{j}\omega(s)\|_{L^1}\leq \|\partial_{i}u(s)\|_\infty \|\nabla\partial_{j}\omega(s)\|_{L^1}
    \\&\qquad\leq \|\nabla u(s)\|_\infty\|\omega(s)\|_{W^{2,1}},
\end{split}
\end{equation}
with the same estimate for the last term. Combining (\ref{twoderivprep}), (\ref{termone}), and (\ref{termtwo}) gives \begin{align}\label{twoderiv}
    \|\partial_{ij}^2 \omega(t)\|_{L^{1}}\leq\|\partial_{ij}^2\omega_0\|_{L^1}+C\int_0^t \|\nabla u(s)\|_\infty\|\omega(s)\|_{W^{2,1}}\, ds.
\end{align}
Summing (\ref{Lpcons}), (\ref{onederiv}), and (\ref{twoderiv}) over $i,j\in \{1,2\}$, we have \begin{equation}\label{globalprep}
    \|\omega(t)\|_{W^{2,1}}\leq\|\omega_0\|_{W^{2,1}}+C\int_0^t\|\nabla u(s)\|_\infty\|\omega(s)\|_{W^{2,1}}ds.
\end{equation}
An application of Lemma \ref{velgradbound} gives \begin{equation*}
\|\omega(t)\|_{W^{2,1}}\leq\|\omega_0\|_{W^{2,1}}+C\int_0^t\|\omega(s)\|_{W^{2,1}}^2\, ds.
\end{equation*} 
\\ Now applying Osgood's Lemma (Lemma \ref{Osgood}) with $\mu(r)=r^2$ gives the desired estimate, 
\begin{equation*}
\|\omega(t)\|_{W^{2,1}}\leq\frac{\|\omega_0\|_{W^{2,1}}}{1-Ct\|\omega_0\|_{W^{2,1}}}.
\end{equation*} To ensure the right hand side is finite, we choose $T<\big(C\|\omega_0\|_{W^{2,1}}\big)^{-1}$. This completes the proof.
\end{proof}
\begin{remark}
The estimate (\ref{pretermone}) above and the Sobolev embedding theorem give
$$\|\partial_{ij}^2u(s)\cdot\nabla\omega(s)\|_{L^1} \leq C\|\nabla\omega(s)\|_{L^2}^2\leq C\|\omega\|^2_{W^{2,1}}.$$
This estimate is sufficient to establish short-time existence of solutions with vorticity in $W^{2,1}(\mathbb{R}^2)$. However, by refining this estimate in (\ref{termone}), we obtain (\ref{globalprep}), which is essential to the proof in Section \ref{section4} of global existence of solutions with vorticity in $W^{2,1}(\mathbb{R}^2)$ under the assumption that the initial vorticity is Dini continuous.   
\end{remark}
\section{Proof of Theorem \ref{Shorttimesec1}}\label{ProofOfMainTheorem}

In this section, we prove Theorem \ref{Shorttimesec1}. Our strategy is to construct an approximating sequence of smooth solutions which exist on a common time interval. We then apply Theorems \ref{DP} and \ref{RFK} to obtain a candidate for $\omega$. We complete the proof by showing that $\omega$ is indeed a weak solution to (\ref{V}), as in Definition \ref{Weakvorticity}.

{\bf The Approximation Sequence and Uniform Bounds.} To begin, let $\rho\in C_0^\infty(\mathbb{R}^2)$ be a nonnegative bump function with $\rho(x)=0$ whenever $|x|>1$ and $\int_{\mathbb{R}^2}\rho=1$. Define $\rho_\varepsilon(x)=\varepsilon^{-2}\rho(x/\varepsilon)$ and $f^\varepsilon= \rho_\varepsilon*f $ for a distribution $f$. We have the following useful facts: \begin{equation}\label{fourfacts}
\begin{split}
    &\quad  \|f^\varepsilon\|_{L^1}\leq \|f\|_{L^1},
    \\&\quad \|f^\varepsilon\|_\infty\leq\|f\|_\infty,
    \\&\quad \partial^\alpha f^\varepsilon=(\partial^\alpha f)^\varepsilon \ \ \text{ for all } \alpha \in \mathbb{Z}_{\geq0}^2,
    \\ &\quad \|f^\varepsilon\|_{W^{k,1}}\leq\|f\|_{W^{k,1}}\ \  \text{ for all } k\in\mathbb{N}. 
\end{split}
\end{equation}
To define our approximating sequence, for each $\varepsilon>0$ we let $\omega^\varepsilon$ denote the global-in-time smooth solution to (\ref{V}) generated from the initial data $\omega_0^\varepsilon$ (noting, however, that it does not necessarily hold $\omega^\varepsilon(t)=\left(\omega(t)\right)^\varepsilon$, where $\omega(t)$ is the Yudovich solution corresponding to $\omega_0$) and let $u^\varepsilon$ be the corresponding velocity obtained by the Biot-Savart law. Let $C>0$ be as in Proposition \ref{Eulerapriori}, and choose $$T<(C\|\omega_0\|_{W^{2,1}})^{-1}.$$  From $(\ref{fourfacts})_4$ and Proposition \ref{Eulerapriori}, it follows that for every $\varepsilon>0$ and every $t\in[0,T]$, 
\begin{align}\label{uniformbound}\|\omega^\varepsilon(t)\|_{W^{2,1}}\leq\frac{\|\omega_0^\varepsilon\|_{W^{2,1}}}{1-Ct\|\omega_0^\varepsilon\|_{W^{2,1}}}\leq\frac{\|\omega_0\|_{W^{2,1}}}{1-CT\|\omega_0\|_{W^{2,1}}}:=M<\infty.\end{align}   
\noindent To obtain a candidate $\omega$ for the weak solution to (\ref{V}), we will apply (\ref{uniformbound}) and Theorems \ref{DP} and \ref{RFK} to pass to the limit as $\varepsilon$ approaches $0$.

In what follows, we will establish strong $W^{1,1}$ convergence and weak $W^{2,1}$ convergence of a sequence of smooth vorticity solutions to (\ref{V}). In view of Theorems \ref{DP} and \ref{RFK}, we must show that, for all $t\in[0,T]$, the families $\{\omega^\varepsilon(t):\varepsilon\in(0,1]\}$ and $\{\partial_i\omega^\varepsilon(t):\varepsilon\in(0,1]\}$, $i=1,2$, are equicontinuous and equitight in $L^1(\mathbb{R}^2)$, and that $\{\partial^2_{ij}\omega^\varepsilon(t):\varepsilon\in(0,1]\}$, $i,j\in\{1,2\}$, is equitight and equiintegrable in $L^1(\mathbb{R}^2)$. Our strategy will make extensive use of the equality $$\omega^\varepsilon(t,x)=\omega_0^\varepsilon(X^{-t}_\varepsilon(x)),$$ which follows from (\ref{characteristics}).  Here $X_\varepsilon^{-t}:\mathbb{R}^2\to\mathbb{R}^2$ is the inverse particle trajectory map defined by $$X_\varepsilon^{-t}(X_\varepsilon(t,\alpha))=\alpha= X_\varepsilon(t,X^{-t}_\varepsilon(\alpha)), \quad t\geq0, \text{ 
 }\alpha\in\mathbb{R}^2.$$ 
The incompressibility of the flow implies $X_\varepsilon^{-t}$ is a volume-preserving homeomorphism for each $t$. Moreover, $X_\varepsilon^{-t}$ satisfies the equalities \begin{align*}
    &X_\varepsilon^{-t}(x)-x=\int_0^t -u^\varepsilon(t-\tau,X_\varepsilon^{\tau-t}(x)) \,d\tau,
    \\ &\partial_iX^{-t}_\varepsilon(x)-e_i=\int_0^t-\nabla u^\varepsilon(t-\tau,X_\varepsilon^{\tau-t}(x))\cdot\partial_iX^{\tau-t}_\varepsilon (x) \, d\tau,
    \\ &\partial_{ij}^2X_\varepsilon^{-t}(x)=\int_0^t\bigg (-\sum_{k=1}^2\partial_j\big(X_\varepsilon^{\tau-t}\big)_k(x)\partial_k\nabla u^\varepsilon(t-\tau,X_\varepsilon^{\tau-t}(x))\cdot\partial_iX_\varepsilon^{\tau-t}(x)
    \\ \notag&\quad \quad \quad\quad\quad\quad\quad \quad \quad-\nabla u^{\varepsilon}(t-\tau,X^{\tau-t}_\varepsilon(x))\cdot\partial_{ij}^2X_\varepsilon^{\tau-t}(x) \bigg)\, d\tau,
\end{align*}
from which we conclude
\begin{equation}
    \label{flowoneboundprep}\|\nabla X_\varepsilon^{-t}\|_\infty \leq 1+\int_0^t\|\nabla u^\varepsilon(t-\tau)\|_\infty\|\nabla X_\varepsilon^{\tau-t}\|_\infty \, d\tau,
\end{equation}
and
\begin{equation}\label{flowtwoboundprep}
\begin{split}
    &\|\nabla^2X_\varepsilon^{-t}\|_{L^2}\leq C\int_0^t\big (\|\nabla^2 u^{\varepsilon}(t-\tau)\|_{L^2}\|\nabla X_\varepsilon^{\tau-t}\|_\infty^2\\
    &\qquad +\|\nabla u^\varepsilon(t-\tau)\|_\infty\|\nabla^2X_\varepsilon^{\tau-t}\|_{L^2} \big)\, d\tau,
    \end{split}
\end{equation}
where $C$ is an absolute constant. Applying Lemma \ref{Gronwall}, Lemma \ref{velgradbound}, and (\ref{uniformbound}) to (\ref{flowoneboundprep}) gives 
\begin{equation}\label{flowonebound}
\|\nabla X_\varepsilon^{-t}\|_\infty\leq e^{\int_0^t\|\nabla u^\varepsilon(t-\tau)\|_\infty d\tau}\leq e^{MT}\end{equation}
for all $t\in[0,T]$. Substituting (\ref{flowonebound}) into (\ref{flowtwoboundprep}) and applying Lemma \ref{Gronwall} again gives 
\begin{equation}\label{flowtwobound}
\begin{split}
&\|\nabla^2X_\varepsilon^{-t}\|_{L^2}\leq C\|\nabla^2u^\varepsilon\|_{L^\infty_tL^2_x}Te^{CMT}e^{C\int_0^t\|\nabla u^\varepsilon(t-\tau)\|_\infty d\tau}\\
&\qquad\qquad  \leq CMTe^{CMT} \leq e^{CMT}.
\end{split}
\end{equation}
To obtain the second-to-last inequality above, we utilized the series of estimates
$$ \|\nabla^2u^\varepsilon(t)\|_{L^2} \leq \|\nabla\omega^{\varepsilon}(t) \|_{L^2} \leq \| \omega^{\varepsilon} (t) \|_{W^{2,1}} \leq M $$
for every $t\in[0,T]$.

With uniform bounds on the sequences of vorticity solutions and corresponding particle trajectory maps, we are now in a position to establish equicontinuity, equitightness, and equiintegrability of the families of smooth solutions and their derivatives.

{\bf Equicontinuity.} We begin by showing equicontinuity of the two families $\{\omega^\varepsilon(t):\varepsilon\in(0,1]\}$ and $\{\partial_i\omega^\varepsilon(t):\varepsilon\in(0,1],i=1,2\}$. First note that for smooth functions $f\in W^{1,1}(\mathbb{R}^2)$, \begin{align*}
    f(x+h)-f(x)=\int_0^{|h|}\nabla f(x+sh/|h|)\cdot h/|h| \,ds.
\end{align*}
Taking an $L^1$ norm gives $$\int_{\mathbb{R}^2}|f(x)-f(x+h)|\, dx\leq\int_0^{|h|}\left \|\nabla f(x+sh/|h|) \right\|_{L^1_x} \, ds=\int_0^{|h|}\|\nabla f\|_{L^1}\, ds=|h|\|\nabla f\|_{L^1},$$ where we used the translation invariance of the $L^1$ norm. It then follows from the uniform bound (\ref{uniformbound}) that $\{\omega^\varepsilon(t):\varepsilon\in(0,1]\}$ and $\{\partial_i\omega^\varepsilon(t):\varepsilon\in(0,1],i=1,2\}$ are equicontinuous in the $L^1$ norm for all $t\in[0,T].$

{\bf Equitightness.} We now establish equitightness of $\{\omega^\varepsilon(t):\varepsilon\in(0,1]\}$, $\{\partial_i\omega^\varepsilon(t):\varepsilon\in(0,1]\}$, and $\{\partial^2_{ij}\omega^\varepsilon(t):\varepsilon\in(0,1]\}$, $i,j\in\{1,2\}$, in $L^1(\mathbb{R}^2)$. First note that $\omega_0\in W^{2,1}(\mathbb{R}^2)$ implies
\begin{equation}\label{decay}\lim_{R\to \infty}\sum_{|\alpha|\leq 2}\int_{|x|\geq R}|\partial^\alpha\omega_0(x)| \, dx=0.\end{equation} 
Moreover, for each $\varepsilon\in(0,1]$ and $\alpha\in \mathbb{Z}_{\geq0}^2$ with $|\alpha|\leq 2$, we have
\begin{equation}\label{largexbound}
\begin{split}
    &\int_{|x|\geq R}|\partial^\alpha \omega_0^\varepsilon(x)| \,dx=\int_{|x|\geq R}|(\partial^\alpha\omega_0)^\varepsilon(x)|\,dx \\
    &\qquad =\int_{|x|\geq R}\left|\int_{\mathbb{R}^2}\rho^\varepsilon(x-y)\partial^\alpha\omega_0(y)\,dy\right|\,dx \\
     &\qquad =\int_{|x|\geq R}\left| \int_{|y|\geq R-1}\rho^\varepsilon(x-y)\partial^\alpha\omega_0(y)\,dy\right|\,dx\\
    &\qquad \leq \int_{|y|\geq R-1}\int_{|x|\geq R} \rho^\varepsilon(x-y)\left|\partial^\alpha\omega_0(y)\right|\,dx\,dy\\
    &\qquad \leq \int_{\mathbb{R}^2}\rho^\varepsilon(x)\,dx\int_{|y|\geq R-1}|\partial^\alpha\omega_0(y)|\,dy,
\end{split}
\end{equation}
where we used that $\text{supp } \rho^\varepsilon \subset B_\varepsilon(0)\subset B_1(0)$, $(\ref{fourfacts})_3$, and Fubini's Theorem. This series of inequalities combined with (\ref{decay}) imply that $\{\partial^\alpha\omega_0^\varepsilon: \varepsilon\in(0,1], |\alpha|\leq 2\}$ is equitight.

In order to obtain equitightness of the families of derivatives at positive times, we must now show that equitightness is preserved under composition with the particle trajectory map.  To achieve this, we utilize (\ref{flowonebound}) and (\ref{flowtwobound}). From the equality $\omega^\varepsilon(t,x)=\omega^\varepsilon_0(X^{-t}_\varepsilon(x))$, we have for all $R\geq 2MT,$ $$\int_{|x|\geq R}|\omega^\varepsilon(t,x)| \,dx\leq \int_{|x|\geq R-2MT}|\omega_0^\varepsilon(x)|\,dx,$$ where we used the finite propagation speed giving $\|u^\varepsilon\|_{L^1_tL^\infty_{x}}\leq2MT$, which follows from the inequalities (see, for example, Proposition 8.2 of \cite{majda}),
\begin{equation}\label{boundedvelocity}
\begin{split}
&\|u^\varepsilon\|_{L^\infty([0,T];L^{\infty}(\mathbb{R}^2))}\leq \|\omega^\varepsilon\|_{L^\infty([0,T];L^{1})}+\|\omega^\varepsilon\|_{L^\infty([0,T];L^{\infty})}\\
&\qquad \qquad \leq 2\|\omega^\varepsilon\|_{L^\infty([0,T]; W^{2,1})} \leq 2M.
\end{split}
\end{equation}
Thus, taking the limit $R\to\infty$, we obtain equitightness of $\{\omega^\varepsilon(t,x):\varepsilon\in(0,1]\}$ by applying (\ref{decay}) and (\ref{largexbound}).

Next, we have for all $R\geq 2MT$ and $i\in\{1,2\}$, 
\begin{equation*}
\begin{split}
&\int_{|x|\geq R}|\partial_i\omega^\varepsilon(t,x)|\,dx = \int_{|x|\geq R}|\nabla\omega^\varepsilon_0(X^{-t}_\varepsilon(x))\cdot \partial_iX_\varepsilon^{-t}(x)|\,dx
\\ &\qquad\qquad\leq \|\nabla X_\varepsilon^{-t}\|_\infty\max_{i=1,2}\int_{|x|\geq R}|\partial_i\omega_0^\varepsilon(X^{-t}_\varepsilon(x))|\,dx
\\&\qquad\qquad \leq e^{MT}\max_{i=1,2}\int_{|x|\geq R-2MT}|\partial_i\omega_0^\varepsilon(x)|\,dx,
\end{split}
\end{equation*}
by (\ref{flowonebound}) and a change of variables. Equitightness of $\{\partial_i\omega^\varepsilon(t):\varepsilon\in(0,1]\}$ follows from passing to the limit as $R\to\infty$ and applying (\ref{decay}) and (\ref{largexbound}).

Finally, consider $i,j\in\{1,2\}$ and let $R\geq 2MT+1$. The chain rule gives 
\begin{equation}\label{chainrule}
\begin{split}
&\partial_{ij}^2\omega^\varepsilon(t,x)=\sum_{k=1}^2\big(\partial_iX_\varepsilon^{-t}\big)_k(x)\partial_k\nabla\omega_0^\varepsilon(X_\varepsilon^{-t}(x))\cdot\partial_jX_\varepsilon^{-t}(x)\\
&\qquad +(\nabla\omega_0^\varepsilon)(X_\varepsilon^{-t}(x))\cdot\partial_{ij}^2\big(X_\varepsilon^{-t}\big)(x),
\end{split}
\end{equation}
so that, by (\ref{flowonebound}), (\ref{flowtwobound}), a change of variables, and H\"older's inequality,
\begin{equation*}
\begin{split}
    &\int_{|x|\geq R}|\partial_{ij}^2\omega^\varepsilon(t,x)|dx \leq C\|\nabla X_\varepsilon^{-t}\|^2_\infty\max_{k=1,2}\int_{|x|\geq R}|\partial_{ik}^2\omega_0^\varepsilon(X_\varepsilon^{-t}(x))| \,dx
    \\&\qquad  +C\|\nabla^2X_\varepsilon^{-t}\|_{L^2}\max_{k=1,2}\Big(\int_{|x|\geq R}|\partial_k\omega_0^\varepsilon(X^{-t}_\varepsilon(x))|^2dx\Big)^{1/2}
    \\& \qquad\leq Ce^{CMT}\max_{k=1,2}\int_{|x|\geq R-2MT}|\partial_{ik}^2\omega_0^\varepsilon(x)|\,dx
    \\ &\qquad +Ce^{CMT}\max_{k=1,2}\Big(\int_{|x|\geq R-2MT}\Big|\int_{|y|\leq1}\partial_k\omega_0(x-y)\rho_\varepsilon(y)\,dy\Big|^2dx\Big)^{1/2}
    \\ & \qquad \leq  Ce^{CMT}\max_{k=1,2}\int_{|x|\geq R-2MT}|\partial_{ik}^2\omega_0^\varepsilon(x)|\,dx
    \\ &\qquad +Ce^{CMT}\max_{k=1,2}\int_{|y|\leq1}\rho^\varepsilon(y)\Big(\int_{|x|\geq R-2MT-1}|\partial_k\omega_0(x)|^2dx\Big)^{1/2}\,dy,
\end{split}
\end{equation*}
where we also used Minkowski's inequality for integrals in the last line. As $R\to\infty$, the first integral of the final expression above vanishes due to the equitightness of the second order derivatives of the mollified initial data. Further, by Sobolev's inequality, $\nabla \omega_0 \in  L^2(\mathbb{R}^2)^2$, the second term in the final expression is finite and also vanishes as $R\to\infty.$ This gives equitightness of $\{\partial_{ij}^2\omega^\varepsilon(t):\varepsilon\in(0,1],i,j\in\{1,2\}\}$.

{\bf Equiintegrability.} We now establish equiintegrability of $\{\partial_{ij}^2\omega^\varepsilon(t):\varepsilon\in(0,1],i,j\in\{1,2\}\}$. Fix $\eta>0$. By the absolute continuity of integration, there exists $\delta>0$ so that for all measurable $A\subset\mathbb{R}^2$ with $|A|<\delta$, one has 
\begin{equation}\label{ei}\int_A |\partial^2_{ij}\omega_0(x)|\, dx<\eta
\end{equation}for all $i,j\in\{1,2\}$. Now note that for each $\varepsilon\in(0,1]$, $i,j\in\{1,2\}$, and $A\subset\mathbb{R}^2$ with $|A|<\delta$, 
\begin{equation}\label{smoothei}
\begin{split}
&\int_A|\partial_{ij}^2\omega_0^\varepsilon(x)|\,dx=\int_A\left|\big(\partial^2_{ij}\omega_0\big)^\varepsilon(x)\right|\,dx\leq\int_A\int_{\mathbb{R}^2}\rho_\varepsilon(y)\left|\partial_{ij}^2\omega_0(x-y)\right|\,dy\,dx
\\ &\qquad =\int_{\mathbb{R}^2}\rho_\varepsilon(y) \left(\int_A|\partial_{ij}^2\omega_0(x-y)|\,dx \right ) \,dy =\int_{A-y}\left|\partial_{ij}^2\omega_0(x)\right|\, dx 
<\eta,\end{split}
\end{equation}
where $A-y:=\{a-y\in\mathbb{R}^2:a\in A\}$, so that $|A-y|=|A|$ for all $y\in\mathbb{R}^2$. The final inequality then follows from (\ref{ei}). We conclude that the second order derivatives of $\omega_0^\varepsilon$ are equiintegrable with respect to $\varepsilon$.

We must now show that equiintegrability is preserved under composition with the particle trajectory map up to time $T$. To achieve this, we again use (\ref{chainrule}) to write,
for $A\subset\mathbb{R}^2$, 
\begin{equation}\label{equi1}   
\begin{split}
&\int_A|\partial_{ij}^2\omega^\varepsilon(t,x)|\,dx \leq C\|\nabla X_\varepsilon^{-t}\|^2_\infty\max_{k=1,2}\int_A|\partial_{ik}^2\omega_0^\varepsilon(X_\varepsilon^{-t}(x))|\,dx
    \\&\qquad+C\|\nabla^2X_\varepsilon^{-t}\|_{L^2}\max_{k=1,2}\left(\int_A|\partial_k\omega_0^\varepsilon(X^{-t}_\varepsilon(x))|^2\,dx\right)^{1/2}.
\end{split}
\end{equation}
Since the inverse particle trajectory maps are measure-preserving, we can make the first term on the right-hand-side of (\ref{equi1}) as small as we would like by choosing a small enough $\delta>0$ with $|A|<\delta$ and using (\ref{smoothei}) and (\ref{flowonebound}). For the integral in the second term on the right hand side, we use Minkowski's inequality for integrals to write \begin{align*}
    &\left(\int_A|\partial_k\omega_0^\varepsilon(X^{-t}_\varepsilon(x))|^2 \,dx\right)^{1/2}=\left(\int_{X^{-t}_\varepsilon(A)}\Big|\int_{\mathbb{R}^2}\partial_k\omega_0(x-y)\rho_\varepsilon(y)\,dy\Big|^2\,dx\right)^{1/2}
    \\&\qquad\qquad \leq\int_{\mathbb{R}^2}\rho_\varepsilon(y) \left(\int_{X^{-t}_\varepsilon(A)-y}|\partial_k\omega_0(x)|^2\,dx\right)^{1/2}\,dy.
    \end{align*} 
Again, we utilize the fact that $X_\varepsilon^{-t}$ is measure-preserving to infer that $|X_\varepsilon^{-t}(A)-y|=|A|$. Thus we can make this integral as small as we would like by choosing $A$ with sufficiently small measure and noting that $|\partial_k\omega_0|^2\in L^1(\mathbb{R}^2)$ for each $k$ by the Sobolev embedding thoerem. We combine this observation with (\ref{flowtwobound}) and our estimate for the first term on the right-hand-side of (\ref{equi1}) to conclude that $\{\partial_{ij}\omega^\varepsilon(t):\varepsilon\in (0,1]\}$ is equiintegrable for each $i,j\in \{1,2\}$.

{\bf Convergence of a Subsequence.} With equicontinuity, equitightness, and equiintegrability in hand, we can now show there exists a sequence $\varepsilon_k\to 0$ such that $\omega^k=\omega^{\varepsilon_k}$ is Cauchy in $L^\infty([0,T];W^{1,1}(\mathbb{R}^2))$. 

First, we use the vorticity equation to write, for all $t\in[0,T]$, 
\begin{equation*}
\begin{split}
\|\partial_t\omega^\varepsilon(t)\|_{W^{1,1}}\leq &\|u^\varepsilon\cdot\nabla\omega^\varepsilon(t)\|_{L^1}+\sum_{i=1}^2\left(\|\partial_iu^\varepsilon\cdot\nabla\omega^\varepsilon(t)\|_{L^1}+\|u^\varepsilon\cdot\nabla\partial_i\omega^\varepsilon(t)\|_{L^1}\right)\\
 \leq & C\|u^\varepsilon(t)\|_{W^{1,\infty}}\|\omega^\varepsilon(t)\|_{W^{2,1}}\leq CM^2,
\end{split}
\end{equation*}
where we used H\"older's inequality, Lemma \ref{velgradbound}, (\ref{uniformbound}), and (\ref{boundedvelocity}). It follows that \begin{equation}\label{liptime}
    \|\omega^\varepsilon(t)-\omega^\varepsilon(s)\|_{W^{1,1}}\leq \int_s^t\|\partial_t\omega^\varepsilon(\tau)\|_{W^{1,1}}d\tau\leq CM^2|t-s|.\end{equation}
Now let $\eta>0$. Construct a partition of $[0,T]$ as $\{t_0=0,t_1,...,t_N=T\}$, satisfying $$t_i-t_{i-1}<\eta(3CM^2)^{-1}.$$ Applying Theorem \ref{RFK}, we let $\varepsilon_k\to0$ be a sequence such that $\omega^k(t_i)=\omega^{\varepsilon_k}(t_i)$ is convergent in $W^{1,1}(\mathbb{R}^2)$ for each $i=1,...,N$. Now choose $K\in\mathbb{N}$ such that, for all $k,j\geq K$, 
$$\|\omega^k(t_i)-\omega^j(t_i)\|_{W^{1,1}}<\eta/3$$ for $i=1,...,N$. Note that for all $t\in[0,T]$, there is some $i\in\{1,...,N\}$ such that $t\in [t_{i-1},t_i]$. It follows that for all $t\in[0,T]$, there exists $i$ such that for all $k,j\geq K$, \begin{equation*}
\begin{split}
    &\|\omega^k(t)-\omega^j(t)\|_{W^{1,1}}\leq\|\omega^k(t)-\omega^k(t_i)\|_{W^{1,1}}+\|\omega^k(t_i)-\omega^j(t_i)\|_{W^{1,1}}
    \\&\quad+\|\omega^j(t_i)-\omega^j(t)\|_{W^{1,1}}<\eta/3+\eta/3+\eta/3=\eta.
    \end{split}
\end{equation*}
Thus $(\omega^k)_k$ is Cauchy in $L^{\infty}([0,T];W^{1,1}(\mathbb{R}^2))$, and hence there exists $\omega$ in this space for which  \begin{equation}\label{w11conv}
\sup_{t\in[0,T]}\|\omega^k(t)-\omega(t)\|_{W^{1,1}}\to0.
\end{equation} 
Moreover, (\ref{liptime}) shows $\omega^k$ is uniformly Lipschitz in time in the $W^{1,1}$-norm with Lipschitz constant $CM^2$. The above convergence then gives 
$$\omega\in \text{Lip}([0,T];W^{1,1}(\mathbb{R}^2)).$$ 
To obtain the $W^{2,1}$-regularity of $\omega$ for all $t\in[0,T]$, we apply Theorem \ref{DP} to conclude that (up to a possibly further subsequence for each $t$) $\omega^k(t)\rightharpoonup \omega(t)$ in $W^{2,1}(\mathbb{R}^2)$. By lower semicontinuity of the norm under weak convergence, it follows that $$\omega\in L^\infty([0,T];W^{2,1}(\mathbb{R}^2))\cap\text{Lip}([0,T];W^{1,1}(\mathbb{R}^2)),$$ as claimed.

{\bf The limit $\omega$ is a weak solution.} It remains to show that $\omega$ is a weak solution to (\ref{V}) as in Definition \ref{Weakvorticity}. First fix any $q\in (2,\infty)$ and note that by the Hardy-Littlewood-Sobolev inequality,
$$\sup_{t\in[0,T]}\|K*\omega^k(t)-K*\omega(t)\|_{L^q}\leq C\sup_{t\in [0,T]}\|\omega^k(t)-\omega(t)\|_{L^p},$$ 
where $p$ satsifies $1/p=1/q+1/2$. The embedding $W^{1,1}(\mathbb{R}^2)\hookrightarrow L^p(\mathbb{R}^2)$ for $p\in[1,2]$ and (\ref{w11conv}) then imply
\begin{equation}\label{h1convu}
    u^k\to u=K*\omega \text{ in }C([0,T];L^q(\mathbb{R}^2)).
\end{equation}

Now let $\varphi\in C^1([0,T];C^1_0(\mathbb{R}^2))$. Since $\omega^k$ solves (\ref{V}) for each $k\in\mathbb{N}$, we have $$\int_{\mathbb{R}^2}\left(\varphi(T)\omega^k(T)-\varphi_0\omega_0^k\right)\,dx=\int_0^T\int_{\mathbb{R}^2}\left(\omega^k\partial_t\varphi+\omega^ku^k\cdot\nabla\varphi \right)\,dx\,dt.$$
Since $\omega^k\to\omega$ in $W^{1,1}(\mathbb{R}^2)$ uniformly in time, the left hand side of the above satisfies $$\int_{\mathbb{R}^2}\left(\varphi(T)\omega^k(T)-\varphi_0\omega_0^k\right)\, dx\to\int_{\mathbb{R}^2}\left(\varphi(T)\omega(T)-\varphi_0\omega_0\right) \, dx,$$ while the first term on the right hand side satisfies $$\int_0^T\int_{\mathbb{R}^2}\omega^k\partial_t\varphi \,dx\,dt\to\int_0^T\int_{\mathbb{R}^2}\omega\partial_t\varphi \,dx\,dt.$$ For the nonlinear term, we use the bound $\|u^k\|_{L^\infty([0,T]\times\mathbb{R}^2)}\leq M$ given in (\ref{boundedvelocity}) and the bound $\|\omega^k\|_{L^\infty([0,T];L^2(\mathbb{R}^2))}\leq M$ to write
\begin{equation*}
\begin{split}
    &\left|\int_0^T\int_{\mathbb{R}^2}(\omega^ku^k-\omega u)\cdot\nabla\varphi \, dx \, dt\right|
    \leq \int_0^T\int_{\mathbb{R}^2}\Big(|\omega||u^k-u|+|\omega^k-\omega||u^k|\Big)|\nabla\varphi| \,dx \, dt
    \\&\qquad \leq MT\|\nabla\varphi\|_{L^\infty\left([0,T];L^{\frac{q}{q-1}}(\mathbb{R}^2)\right)} \|u^k-u\|_{L^\infty([0,T];L^q(\mathbb{R}^2))}
    \\ &\qquad +MT\|\nabla\varphi\|_{L^\infty([0,T]\times \mathbb{R}^2)}\|\omega^k-\omega\|_{L^\infty([0,T];L^1(\mathbb{R}^2))},
\end{split}
\end{equation*}
which tends to 0 by (\ref{w11conv}) and (\ref{h1convu}). Thus $\omega$ is a weak solution to (\ref{E}) as in Definition \ref{Weakvorticity}.
\section{Global Well-Posedness with Dini Continuous Vorticity}\label{section4}
In this section, we assume that the initial vorticity belongs to $W^{2,1}(\mathbb{R}^2)$ and is Dini continuous (see below), and we show that the solution remains in $W^{2,1}(\mathbb{R}^2)$ for all time. 

Recall the definition of the Dini seminorm for continuous functions $f:\mathbb{R}^d\to\mathbb{R}$, given by
$$|f|_{C_D}:=\int_0^1\sup_{|x-y|\leq r}|f(x)-f(y)| \, \frac{dr}{r}.$$ We define the space of {\em Dini continuous} functions to be the set of $f\in L^{\infty}$ for which $|f|_{C_D} < \infty$, and we denote the space of Dini continuous functions by $C_D$. We prove the following theorem. 
\begin{theorem}\label{Globaltheorem} If $\omega_0\in C_D\cap W^{2,1}(\mathbb{R}^2)$, then the solution $\omega$ obtained in Theorem \ref{Shorttimesec1} belongs to $L^\infty_{loc}(\mathbb{R}_+;C_D\cap W^{2,1} (\mathbb{R}^2))$. Furthermore, $\|\omega(t)\|_{W^{2,1}\cap C_D}:=\|\omega(t)\|_{W^{2,1}}+|\omega(t)|_{C_D}$ grows at most doubly exponentially with $t$.
\end{theorem}
  \noindent Note that the Dini continuity assumption in Theorem \ref{Globaltheorem} is stronger than the uniform continuity already satisfied by $W^{2,1}(\mathbb{R}^2)$ functions, but is weaker than an $\alpha-$H\"older condition with $\alpha>0$, for which global existence is established in \cite{wolibner1933} (see \cite{chemin} for another approach). 
  
  Our strategy for proving Theorem \ref{Globaltheorem} is to establish a bound on the $L^{\infty}$-norm of the velocity gradient which grows at most exponentially in time, provided the initial vorticity is Dini continuous. We then apply this bound to (\ref{globalprep}). A similar bound on the velocity gradient is established by Koch \cite{koch} on a bounded domain; specifically, the author shows that for $\Omega$ bounded,
  \begin{equation}\label{kochbound}
  \|\nabla u(t) \|_{L^{\infty}(\Omega)} \leq C(\| \omega(t) \|_{C_D(\Omega)} + \| \omega(t) \|_{L^{\infty}(\Omega)}).
  \end{equation}
  The author then proceeds to show that the Dini modulus of continuity is preserved by the Euler equations. 
  
  We will follow the methods in \cite{koch}, but the arguments there and the estimate (\ref{kochbound}) in particular require some modification for our setting. Therefore, for the sake of completeness, we include our full argument here. 

Before proving Theorem \ref{Globaltheorem} we establish the following proposition, which is the key ingredient of the proof. 
\begin{prop}\label{velgraddiniprop}
     Let $\omega$ be a smooth solution to $(\ref{V})$ with initial data $\omega_0$, and let $X$ be the particle trajectory map corresponding to $u$. Then there exists $C>0$ independent of time and $\omega_0$ such that, for all $t\geq0$, \begin{equation}\label{velgraddini} 
     \|\nabla u(t)\|_\infty \leq C\Big(\|\omega_0\|_{L^1}+\|\omega_0\|_\infty+|\omega_0|_{C_D}\Big)e^{C\|\omega_0\|_{\infty}t}.
     \end{equation}
\end{prop}
\begin{proof}
Let $\varphi \in C_0^\infty(\mathbb{R}^2)$ be a radial function such that $\varphi(x)\equiv1$ for $0\leq |x|\leq \frac{1}{2}$, $\varphi(x)=0$ for $|x|\geq 1$, and $|\nabla\varphi(x)|\leq 4$.  We utilize the Biot-Savart law to write \begin{equation}\label{nablaudecomp}
\begin{split}
    \nabla u(t,x)=\frac{1}{2\pi}&\int_{|x-y|\leq1}\frac{(x-y)^\perp}{|x-y|^2}\varphi(x-y)\nabla\omega(t,y)\, dy
    \\ &+\frac{1}{2\pi}\int_{|x-y|\geq\frac{1}{2}}\frac{(x-y)^\perp}{|x-y|^2}\big(1-\varphi(x-y)\big)\nabla\omega(t,y)\, dy,
\end{split}    
\end{equation}
and we estimate each term in the sum independently.

For the first term, integration by parts and a change of coordinates give
\begin{equation}\label{neartermpart1}
\begin{split}
     &\left|\frac{1}{2\pi}\int_{|x-y|\leq 1}\frac{(x-y)^\perp}{|x-y|^2}\varphi(x-y)\nabla_y\big(\omega(t,y)-\omega(t,x)\big)\,dy\right|
    \\ &=\left|\frac{1}{2\pi}\int_{|x-y|\leq 1}\nabla_y\left(\frac{(x-y)^\perp}{|x-y|^2}\varphi(x-y)\right)\big(\omega(t,y)-\omega(t,x)\big)\,dy\right|
    \\ &\leq C\int_{|x-y|\leq 1}\frac{1}{|x-y|^2}\big|\omega(t,y)-\omega(t,x)\big| \, dy
    \\  &\qquad +C\int_{\frac{1}{2}\leq|x-y|\leq 1}\frac{1}{|x-y|}|\omega(t,y)-\omega(t,x)|\, dy
    \\ &\leq  C\int_0^1\sup_{|x-y|\leq r}|\omega(t,x)-\omega(t,y)|\frac{dr}{r}+C\int_{\frac{1}{2}}^12\|\omega(t)\|_\infty\, dr
    \\  &\leq C\int_0^{1}\sup_{|x-y|\leq r}|\omega(t,x)-\omega(t,y)|\, \frac{dr}{r}+C\|\omega_0\|_\infty,
\end{split}
\end{equation}
where in the last inequality, we used conservation of the $L^\infty$ norm of the vorticity. To estimate the term
$$ C\int_0^{1}\sup_{|x-y|\leq r}|\omega(t,x)-\omega(t,y)|\, \frac{dr}{r},$$
we let $X^{-t}$ denote the inverse particle trajectory map corresponding to the velocity $u$ and use the conservation of vorticity along particle trajectories to obtain, 
\begin{equation}\label{Phi0part2}
\begin{split}
    &C\int_0^{1}\sup_{|x-y|\leq r}|\omega(t,x)-\omega(t,y)|\, \frac{dr}{r}\\
    &\qquad=C\int_0^{1}\sup_{|x-y|\leq r}|\omega_0(X^{-t}(x))-\omega_0(X^{-t}(y))|\, \frac{dr}{r}
    \\ &\qquad \leq C\int_0^{1}\sup_{|x-y|\leq r\|\nabla X^{-t}\|_\infty}|\omega_0(x)-\omega_0(y)|\, \frac{dr}{r}
    \\  &\qquad   =C\int_0^{\|\nabla X^{-t}\|_\infty}\sup_{|x-y|\leq r}|\omega_0(x)-\omega_0(y)|\, \frac{dr}{r}
    \\ &\qquad =C|\omega_0|_{C_D}+C\int_1^{\|\nabla X^{-t}\|_\infty}\sup_{|x-y|\leq r}|\omega_0(x)-\omega_0(y)|\, \frac{dr}{r}.
    \end{split}
    \end{equation}
Now observe that
\begin{equation}\label{Phi0part3}
\begin{split}
    &\int_1^{\|\nabla X^{-t}\|_\infty}\sup_{|x-y|\leq r}|\omega_0(x)-\omega_0(y)|\, \frac{dr}{r} \leq C\int_1^{\|\nabla X^{-t}\|_\infty}\|\omega_0\|_\infty \frac{dr}{r} 
    \\&\qquad = C\|\omega_0\|_\infty\ln\left( \|\nabla X^{-t}\|_\infty\right) \leq C\|\omega_0\|_\infty\int_0^t\|\nabla u(s)\|_\infty\, ds
    ,
\end{split}
\end{equation}
where the last inequality follows from the calculations leading to (\ref{flowonebound}). Substituting (\ref{Phi0part3}) into (\ref{Phi0part2}) and substituting the resulting bound into $(\ref{neartermpart1})$ gives 
\begin{equation}\label{nearterm}
\begin{split}
    &\left|\frac{1}{2\pi}\int_{|x-y|\leq 1}\frac{(x-y)^\perp}{|x-y|^2}\Phi_0(x-y)\nabla_y\omega(t,y)\, dy\right| \\
    &\qquad \leq C\left(|\omega_0|_{C_D}+\|\omega_0\|_\infty+\|\omega_0\|_\infty\int_0^t\|\nabla u(s)\|_\infty\, ds\right).
    \end{split}
\end{equation}

We now estimate the second term in (\ref{nablaudecomp}) as follows: \begin{equation}\label{farterm}
\begin{split}
    &\qquad \qquad \left|\frac{1}{2\pi}\int_{\frac{1}{2}\leq|x-y|}\frac{(x-y)^\perp}{|x-y|^2}(1-\varphi(x-y))\nabla_y\omega(t,y) \, dy\right| \\
    & \leq C\int_{\frac{1}{2}\leq|x-y|}\left(\frac{1}{|x-y|^2}+\frac{1}{|x-y|}\right)|\omega(t,y)| \, dy
\leq C\|\omega(t)\|_{L^1} = C\|\omega_0\|_{L^1},
\end{split}
\end{equation}
where we used conservation of the $L^1$-norm of the vorticity to get the final equality.

Combining (\ref{nablaudecomp}), (\ref{nearterm}), and (\ref{farterm}), we have 
\begin{equation*}
    \|\nabla u(t)\|_\infty\leq C\Big(\|\omega_0\|_{L^1} +\|\omega_0\|_\infty +|\omega_0|_{C_D}\Big)+C\|\omega_0\|_\infty\int_0^t\|\nabla u(s)\|_\infty \, ds.
\end{equation*}
An application of Lemma \ref{Gronwall} gives 
\begin{equation*}
    \|\nabla u(t)\|_\infty \leq C\Big(\|\omega_0\|_{L^1}+\|\omega_0\|_\infty+|\omega_0|_{C_D}\Big)e^{C\|\omega_0\|_{\infty}t},
\end{equation*}
which is $(\ref{velgraddini})$.
\end{proof} 

We conclude this section with the proof of Theorem \ref{Globaltheorem}. 
\begin{proof}[(Proof of Theorem \ref{Globaltheorem})] Suppose $[0,T^*)$ is the maximal time interval on which the solution obtained in Theorem \ref{Shorttimesec1} persists in $W^{2,1}(\mathbb{R}^d).$ Then we have $T^*=\infty$ or $$\limsup_{t\to T^*}\|\omega(t)\|_{W^{2,1}}=\infty.$$ 
To complete the proof, we will show for all $T<\infty$ that $\limsup_{t\to T}\|\omega(t)\|_{W^{2,1}}<\infty$. To this end, note that $(\ref{globalprep})$ implies 
$$\|\omega^\varepsilon(t)\|_{W^{2,1}}\leq \|\omega_0\|_{W^{2,1}}+C\int_0^t\|\nabla u^\varepsilon(s)\|_\infty\|\omega^\varepsilon(s)\|_{W^{2,1}} \, ds 
$$ for all $t\geq 0$, where $\omega^\varepsilon$ is as in the proof of Theorem \ref{Shorttimesec1}. This estimate, together with Proposition \ref{velgraddiniprop}, gives 
\begin{equation}\label{epsstepone}
\begin{split}
    &\|\omega^\varepsilon(t)\|_{W^{2,1}}\leq \|\omega_0\|_{W^{2,1}}\\
&\qquad +C\int_0^t\Big(\|\omega_0^\varepsilon\|_{L^1}+\|\omega_0^\varepsilon\|_\infty+|\omega_0^\varepsilon|_{C_D}\Big)e^{C\|\omega_0^\varepsilon\|_{\infty}s}\|\omega^\varepsilon(s)\|_{W^{2,1}}\, ds.
\end{split}
\end{equation}
An application of $(\ref{fourfacts})_1$ and $(\ref{fourfacts})_2$ gives $$\|\omega_0^\varepsilon\|_\infty\leq \|\omega_0\|_\infty \text{  and  } \|\omega_0^\varepsilon\|_{L^1}\leq \|\omega_0\|_{L^1}.$$ Moreover, for each $r\in (0,1)$,
\begin{equation*}
\begin{split}
    &\sup_{|x-y|\leq r}|\omega_0^\varepsilon(x)-\omega_0^\varepsilon(y)|=\sup_{|x-y|\leq r}\int_{\mathbb{R}^2}\rho^\varepsilon(z)|\omega_0(x-z)-\omega_0(y-z)| \,dz
    \\&\qquad\qquad  \leq \int_{\mathbb{R}^2}\rho^\varepsilon(z)\sup_{|x-y|\leq r}|\omega_0(x-z)-\omega_0(y-z)|\, dz
    \\&\qquad\qquad  =\sup_{|x-y|\leq r}|\omega_0(x)-\omega_0(y)|,
\end{split}
\end{equation*}
    and hence 
    \begin{equation*}
        |\omega_0^\varepsilon|_{C_D}=\int_0^1\sup_{|x-y|\leq r}|\omega_0^\varepsilon(x)-\omega_0^\varepsilon(y)|\frac{dr}{r}\leq\int_0^1\sup_{|x-y|\leq r}|\omega_0(x)-\omega(y)|\frac{dr}{r}=|\omega_0|_{C_D}.
    \end{equation*}
Applying these bounds to (\ref{epsstepone}) gives 
\begin{equation*}
    \|\omega^\varepsilon(t)\|_{W^{2,1}}\leq \|\omega_0\|_{W^{2,1}}+C\int_0^t\Big(\|\omega_0\|_{L^1}+\|\omega_0\|_\infty+|\omega_0|_{C_D}\Big)e^{C\|\omega_0\|_{\infty}s}\|\omega^\varepsilon(s)\|_{W^{2,1}}\, ds.
\end{equation*}
Lemma \ref{Gronwall} implies 
\begin{equation}\label{epsdoubleexp}
\begin{split}
&\|\omega^\varepsilon(t)\|_{W^{2,1}}\leq \|\omega_0\|_{W^{2,1}}\exp\left\{C\Big(\|\omega_0\|_{L^1}+\|\omega_0\|_\infty+|\omega_0|_{C_D}\Big)\int_0^te^{C\|\omega_0\|_\infty s}\,ds\right\}\\
&\qquad \leq \|\omega_0\|_{W^{2,1}}\exp\Big\{\|\omega_0\|_\infty^{-1}\Big(\|\omega_0\|_{L^1}+\|\omega_0\|_\infty+|\omega_0|_{C_D}\Big)\exp\big\{C\|\omega_0\|_\infty t\big\}\Big\},
\end{split}
\end{equation} 
where $C$ is independent of $\varepsilon,$, giving at most double exponential growth of the $W^{2,1}$-norms of the approximate solutions. Following the proof of Theorem \ref{Shorttimesec1}, we find a subsequence of the sequence $\{\varepsilon_k\}_k$ for which $\varepsilon_k\to 0$ and such that for every $t\in[0,T^*)$, $\omega^{\varepsilon_k}(t)$ converges weakly to $\omega(t)$ in $W^{2,1}(\mathbb{R}^2)$. By restricting ourselves to the sequence obtained in that proof, we guarantee the $\omega$ obtained in the limit is identical on $[0,T]$ to that obtained in the proof of Theorem \ref{Shorttimesec1}. Exploiting the lower semicontinuity of the norm with respect to weak convergence, we obtain the same double exponential growth for $\|\omega(t)\|_{W^{2,1}}$ as in (\ref{epsdoubleexp}); that is, 
\begin{equation*}
\|\omega(t)\|_{W^{2,1}}\leq \|\omega_0\|_{W^{2,1}}\exp\Big\{\|\omega_0\|_\infty^{-1}\Big(\|\omega_0\|_{L^1}+\|\omega_0\|_\infty+|\omega_0|_{C_D}\Big)\exp\big\{C\omega_0\|_\infty t\big\}\Big\}.
\end{equation*}
Hence, $\lim_{t\to T^*}\|\omega(t)\|_{W^{2,1}}=\infty$ only if $T^*=\infty.$ This completes the proof.
\end{proof}
\section*{Acknowledgements}
\noindent The authors thank Tarek Elgindi and Misha Vishik for useful conversations.  EC gratefully acknowledges support by the Simons Foundation through Grant No. 429578.  
\bibliographystyle{acm}
\bibliography{Bib}

\end{document}